\documentclass [11pt,a4paper]{amsart}
\usepackage{amsmath}
\usepackage{stmaryrd}
\usepackage{textcomp}
\usepackage{amsmath}
\usepackage{cases}
\usepackage{amscd}
\usepackage{epsfig}
\usepackage{graphicx}

\usepackage{color}

\usepackage{verbatim}

\usepackage{amssymb}
\usepackage{amsfonts}
\usepackage{amsbsy}

\usepackage{amsthm}
\usepackage{amstext}
\usepackage{amsopn}
\usepackage{hyperref}
\hypersetup{colorlinks=true,
 linkcolor=blue,
filecolor=magenta,urlcolor=blue}

\newcommand{\calH}{\mathcal{H}}
\newcommand{\calS}{\mathcal{S}}
\newcommand{\calC}{\mathcal{C}}

\newcommand{\calU}{\mathcal{U}}

\newcommand{\C}{\mathbb{C}}
\newcommand{\R}{\mathbb{R}}

\newcommand{\J}{\mathbb{J}}
\newcommand{\I}{\mathbb{I}}

\newcommand{\frakH}{\mathfrak{H}}

\newcommand{\bH}{{\bf H}}
\newtheorem{thm}{Theorem}[section]
\newtheorem{defi}[thm]{Definition}
\newtheorem{cor}[thm]{Corollary}
\newtheorem{lem}[thm]{Lemma}
\newtheorem{prop}[thm]{Proposition}

\usepackage{fancyhdr}
\begin{document}
\title[K\"ahler metrics in the Siegel domain]{{\rm PCR} K\"ahler equivalent metrics  in the Siegel domain
}
\author{Joonhyung Kim \& Ioannis D. Platis \& Li-Jie Sun}

\address{Department of Mathematics Education,
Chungnam National University,
 99 Daehak-ro, Yuseong-gu, Daejeon 34134,
 Korea.}
\email{calvary@cnu.ac.kr}

\address{Department of Mathematics and Applied Mathematics,
University of Crete,
 Heraklion Crete 70013,
 Greece.}
\email{jplatis@math.uoc.gr}
\address {Department of applied science, Yamaguchi University,
2-16-1 Tokiwadai, Ube 7558611,
Japan.
}
\email{ljsun@yamaguchi-u.ac.jp}


\thanks{{\it Acknowledgments.}  Parts of this work have been carried out while IDP was visiting Hunan University, Changsha, PRC and JK was visiting University of Crete, Greece. Hospitality is gratefully appreciated. JK was supported by the NRF grant NRF-2020R1F1A1A01050461}
\thanks{\today}
\begin{abstract}

Let $\frakH$ be the Heisenberg group. From the standard CR structure $\calH$ of $\frakH$ we construct the complex hyperbolic structure of the Siegel domain. Additionally, using the same minimal data for $\frakH$, that is, its Sasakian structure, we provide the Siegel domain with yet another K\"ahler structure: this structure is of unbounded negative sectional curvature, and its complex structure does not commute with the standard complex structure. However, we show that those two K\"ahler structures are {\rm PCR} K\"ahler equivalent, that is to say, essentially the same when restricted to $\calH$. 
\end{abstract}

\maketitle

%\tableofcontents

\section{Introduction}
Roughly speaking, there is neither an obvious nor a natural way to construct a K\"ahler metric in a manifold from a CR structure given on its boundary. Obstacles to the solution to this problem may come from various directions; the topology of the manifold plays a rather important role. In this paper we deal with this problem in the case of a topologically "nice" manifold, that is the Siegel domain in $\C^2$. The Siegel domain
$$
\calS=\{(z_1,z_2)\in\C^2:\;2\Re(z_1)+|z_2|^2<0\},
$$
is the underlying manifold of the complex hyperbolic plane $\bH^2_\C$, the latter being $\calS$
equipped with the Bergman metric, see Section \ref{sec:chp} for details. $\bH^2_\C$ is a K\"ahler manifold, with constant holomorphic sectional curvature $-1$ and real sectional curvature pinched between $-1$ and $-1/4.$ Its group of holomorphic isometries is ${\rm PU}(2,1)$ which is a triple cover of ${\rm SU}(2,1)$. The complex hyperbolic plane equipped with this K\"ahler metric is being widely studied from many aspects, see for example \cite{Go, Ka-Re, Pla-CR} and many others.

The topological boundary $\partial\bH^2_\C$ of $\bH^2_\C$ can be identified to the one point compactification of the Heisenberg group $\frakH$ and it  plays a vital and important role in the study of complex hyperbolic geometry. Recall that $\frakH$ is the 2-step nilpotent Lie group with underlying manifold  $\C\times\R$ and multiplication law given by
$$
(z,t)*(w,s)=\left(z+w,\;t+s+2\Im(z\overline{w})\right).
$$ 
The Heisenberg group $\frakH$ is a contact manifold, its contact form being
$$
\omega=dt+2\Im(\overline{z}dz).
$$
As such, it is also a CR-manifold, its CR-structure is $\calH=\ker(\omega)$ (see Sections \ref{sec:prel} and in particular \ref{sec:h}).

In the context of this paper, starting from the Heisenberg group and using the least possible data about it, that is, its CR structure, we are able to make two constructions: firstly, we derive the complex hyperbolic structure of the Siegel domain from a K\"ahler structure of constant negative holomorphic sectional curvature on the Lie group $\frakH\times\R_{>0}$ (the cone of $\frakH$); this is actually the horospherical model of $\bH^2_\C$, see Section \ref{sec-horo}.

Secondly, starting now from Riemannian metrics for $\frakH$ introduced in \cite{CDPT}, we prove that there is only one of them which is Sasakian, see Section \ref{sec:prel} for the definition. From the sub-Riemannian geometry point of view, the Carnot-Carath\'{e}odory metric of $\frakH$ can be viewed as the Gromov-Hausdorff limit of a sequence of Riemannian approximates $g_L,$ which could be identified as an anisotropic blow-up of the Riemannian metric $g=g_{cc}+\eta\otimes\eta,$ see Section \ref{sec:prel} for more details.
Here, we construct a Sasakian structure from the Riemannian approximates $g_L$, see  Section \ref{sec-rapprox}, recovering in this way the well known standard Sasakian structure of $\frakH$. This metric appears to have been discovered long ago by Sasaki \cite{Sas} and Tanno  \cite{Tan}, although the relation with the Heisenberg group was not noticed at that time. This happened a bit later, see \cite{Bo-Ga, BGO}. Moreover,  its appearance in CR spherical geometry was studied further in \cite{Kam}.

 The Sasakian structure on $\frakH$ induces a natural K\"ahler structure on $\frakH\times\R_{>0}$, which is different from the complex hyperbolic K\"ahler structure. In fact, that K\"ahler metric is a warped product metric, and there is no isometric minimal immersion into any complex hyperbolic space, see Section \ref{sec:prel} for the definition and details. 

However, it shows that the Siegel domain with the constructed warped product metric is K\"ahler with unbounded negative curvature, see Corollary \ref{cor-cur}. More importantly, we compare our obtained K\"ahler metric in the Siegel domain with the natural K\"ahler metric of the complex hyperbolic plane and prove that they are {\rm PCR}-K\"ahler equivalent,  see Section \ref{sec:Kahler} for the precise definition.
In fact, what actually happens is that complex structures do not commute but they coincide when restricted to $\calH$. This is within the concept of {\rm PCR}-mappings, see  Section \ref{sec-PCR}, and we show that there are at least two K\"ahler structures in the {\rm PCR}-K\"ahler equivalence class of the Siegel domain, see Theorem \ref{thm-PCRK1} as well as Corollary \ref{cor-PCRK2}. 

The paper is organised as follows: in Section \ref{sec:prel}, we review some standard facts on CR structures, Sasakian structures, sub-Riemannian geometry and warped products. Section \ref{sec-heis} provides a detailed exposition of a Sasakian structure on the Heisenberg group from its Riemannian approximates. 
In Section \ref{sec-Kahler}, the two K\"ahler structures of the Siegel domain are clarified from the viewpoint of the Riemannian cone of the Heisenberg group. Section \ref{sec-PCR} is mainly intended to exhibit that the obtained two K\"ahler manifolds are {\rm PCR}-K\"ahler equivalent. Finally, in Section \ref{sec:chp}, we recall the Siegel domain model of the complex hyperbolic plane and prove that
one of the K\"ahler manifolds we obtained is holomorphically isometric to the complex hyperbolic plane via the horospherical map, see \ref{sec-horo}. .
Finally, geodesics on the  Heisenberg group and its Riemannian cone are presented in Section \ref{sec-app}.

\section{Preliminaries}\label{sec:prel}
The material of this section is standard: See for instance \cite{Be}, \cite{Bo-Ga}, for further details.

%\subsection{{\rm CR}, contact and  Sasakian structures}\label{sec:CRetc
{\bf $\rm{CR}$ structures.} A codimension $s$ {\rm CR} structure in $(2p+s)$-dimensional real manifold $M$ is a pair $(\calH, J)$ where $\calH$ is a $2p$-dimensional smooth subbundle of ${\rm T}(M)$ and $J$ is an almost complex endomorphism of $\calH$ which is formally integrable: that is,  if $X$ and $Y$ are sections of $\calH$ then the same holds for 
  $\left[X, Y\right]-\left[JX, JY\right], \left[JX, Y\right]+\left[X, JY\right]$ and moreover, 
  $J(\left[X, Y\right]-\left[JX, JY\right])=\left[JX, Y\right]+\left[X, JY\right]$.
  
  \medskip

{\bf Contact structures.} If $s=1$, a contact structure on $M$ is a codimension 1 subbundle $\calH$ of ${\rm T}(M)$ which is completely non-integrable; alternatively, $\calH$ may be defined as the kernel of a 1-form $\eta$, called the contact form of $M$, such that $\eta\wedge (d\eta)^{p}\neq 0$.  $\calH$ depends on $\eta$ up to multiplication of $\eta$ by a nowhere vanishing smooth function. By choosing an almost complex structure $J$ defined in $\calH$ we obtain a {\rm CR} structure $(\calH, J)$ of codimension 1 in $M$. The subbundle $\calH$ is also called the horizontal subbundle of ${\rm T}(M)$. The closed form $d\eta$ endows $\calH$ with a symplectic structure and we may demand from $J$ to be such that $d\eta(X,JX)>0$ for each $X\in\calH$; we then say that $\calH$ is strictly pseudoconvex. The Reeb vector field $\xi$ is the vector field which satisfies
$
\eta(\xi)=1$ and $\xi\in\ker(d\eta).
$
Note that the Reeb vector field is uniquely determined by the contact form.
%By the contact version of Darboux's Theorem, $\xi$ is unique up to change of coordinates.

\medskip

{\bf Strictly pseudoconvex domains.}
 Strictly pseudoconvex {\rm CR} structures on boundaries of domains in $\C^2$ are the most illustrative examples of contact structures on 3-dimensional  manifolds. Let $D\subset \C^2$ be a domain with defining function $\rho:D\to\R_{>0}$, $\rho=\rho(z_1,z_2)$. On the boundary $M=\partial D$ we consider the form $d\rho$; if $J$ is the complex structure of $\C^2$ we then let
$$
\eta=-\Im(\partial\rho)=-\frac{1}{2}Jd\rho.
$$
We thus obtain the {\rm CR} structure $(\calH=\ker(\eta),J)$.  This is a contact structure if and only if the Levi form 
$
L=d\eta=i\partial\overline{\partial}\rho
$
is positively oriented.

\medskip

{\bf Contact metric structures.}
Let $(M, \eta)$ be a $(2p+1)$-dimensional pseudo-hermitian manifold equipped with a CR structure $(\calH=\ker(\eta),J).$ 
The almost complex structure $J$ on $\calH$ is then extended to an endomorphism $\phi$ of the whole tangent bundle ${\rm T}(M)$ by setting $\phi(\xi)=0$. Subsequently, a canonical Riemannian metric $g$ is defined in $M$ from the relations
\begin{equation}\label{eq:contactmetric}
\eta(X)=g(X,\xi),\quad \frac{1}{2}d\eta(X,Y)=g(\phi X, Y),\quad \phi^2(X)=-X+\eta(X)\xi,
\end{equation}
for all vector fields $X, Y$ in ${\mathfrak X}(M)$.
We then call $(M;\eta,\xi,\phi,g)$ the contact Riemannian structure  on $M$ associated to the pseudo-hermitian structure $(M, \eta).$ If $f:M\to M$ is an automorphism which preserves the contact Riemannian structure, $f^{*}\eta=\eta$, then one may use equations (\ref{eq:contactmetric}) to verify straightforwardly that this happens if and only if $f$ is {\rm CR}, that is $f_*J=Jf_*$.

\medskip

{\bf Sasakian structures.}
A contact Riemannian manifold for which the Reeb vector field $\xi$ is Killing (equivalently, $\xi$ is an infinitesimal {\rm CR} transformation) is called a  K-contact Riemannian manifold. 

Consider now the Riemannian cone $\calC(M)=(M\times\R_{>0},\;g_r=dr^2+r^2g)$. We may define an almost complex structure $\J$ in $\calC(M)$ by setting
$$
\J X=JX,\quad X\in\calH(M),\quad \J(r\partial_r)=\xi.
$$
The fundamental 2-form for $\calC(M)$ is then the exact form
$$
\Omega_r=d\left(\frac{r^2}{2}\eta\right)=r\;dr\wedge\eta+\frac{r^2}{2}\;d\eta,
$$
and therefore it is closed. We have then that $(M;\eta,\xi,\phi,g)$ is Sasakian if and only if the {\it Riemannian cone} $(\calC(M);\J,g_r,\Omega_r)$ is K\"ahler. The following proposition is often useful.
\begin{prop}\label{prop:S-c}
Let $(\eta,\xi,\phi,g)$ be a K-contact Riemannian structure on $M$. Then $M$ is a Sasakian manifold if and only if the contact Riemannian structure satisfies
$$
R(X,\xi)Y=g(X,Y)\xi-g(\xi,Y)X,
$$
for any vector fields $X,Y$ in $\frak{X}(M)$. Here
%$\nabla$ is the Riemannian connection and 
$R$ is the Riemannian curvature tensor of $g$.
\end{prop}

\medskip

{\bf Sub-Riemannian geometry.}
The sub-Riemannian geometry of a contact (and a contact Riemannian) is described in what follows. If $(M, \eta)$ is a $(2p+1)$-dimensional pseudo-hermitian manifold equipped with a CR structure $(\calH=\ker(\eta),J)$, we define a Riemannian metric $g_{cc}$ in $\calH$ (the sub-Riemannian metric);  the distance $d_{cc}(p,q)$ between two points $p,q$ of $M$ is given by the infimum of the $g_{cc}$-length of horizontal curves joining $p$ and $q$. By a horizontal curve $\gamma$ we mean a piece-wise smooth curve in $M$ such that $\dot\gamma\in\calH$. The metric $d_{cc}$ is the Carnot-Carath\'eodory metric and there are two interesting facts about it: firstly, the metric topology coincides with the manifold topology and secondly, if $g_{cc}'$ is another sub-Riemannain metric, then $d_{cc}$ and $d_{cc}'$ are bi-Lipschitz equivalent on compact subsets of $M$. In the case where we construct a contact Riemannian structure $(M;\eta,\xi,\phi,g)$ out of a pseudo-hermitian structure $(M, \eta)$ as above, the sub-Riemannian metric $g_{cc}$ may be taken as the restriction of $g$ into $\calH\times\calH$, i.e., $g=g_{cc}+\eta\otimes\eta$. If $d_g$ is the Riemannian distance corresponding to the Riemannian metric $g$ and $d_{cc}$ is the Carnot-Carath\'eodory distance corresponding to $g_{cc}$, then we always have $d_g\le d_{cc}$. It also follows that the group ${\rm Aut}(M)$ of automorphisms of the contact Riemannian structure $g$ is just the group ${\rm Isom}_{cc}(M)$ of isometries of $d_{cc}$. If the contact Riemannian structure is Sasakian, then the group ${\rm Aut}(\calC(M))$ of automorphisms of $\calC(M)$ is just ${\rm Isom}_{cc}(M)$.

\medskip

{\bf Warped products.} Let $M_1$ and $M_2$ be two pseudo-Riemannian manifolds equipped with pseudo-Riemannian metrics $g_1$ and $g_2,$ respectively, and let $f$ be a positive smooth function on $M_1.$ Consider the product manifold $M_1\times M_2$ with the following natural projections:
$$\pi^1: M_1\times M_2\to M_1,\qquad \pi^2: M_1\times M_2\to M_2.$$
Then the warped product $M=M_1\times_f M_2$ is the manifold $M_1\times M_2$ equipped with the pseudo-Riemannian structures such that 
$$\langle X, X \rangle=\langle \pi^1_\ast(X), \pi^1_\ast(X)\rangle
+f^2(\pi^1(x))\langle \pi^2_\ast(X), \pi^2_\ast(X)\rangle$$
for any tangent vector $X\in TM.$ Thus one can have $g_M=g_1+f^2g_2,$ where the function $f$ is named the warping function of the warped product.

%(i) trivial, if $f$ is constant; 
Chen in \cite{CH} stated that if $\phi: M_1\times_f M_2\to {\textbf{H}_{\mathbb{C}}^n}$ is an isometric immersion of a warped product $M_1\times_f M_2$ into the complex hyperbolic $n$-space with constant holomorphic sectional curvature $-4,$ then one can get that
$${\frac{\Delta f}{f}\leq \frac{(m_1+m_2)^2}{4m_2}H^2-m_1,}$$
where $m_i={\rm dim} M_i, i=1, 2,$  $H^2$ is the squared mean curvature of $\phi$, and $\Delta$ is the Laplacian operator of $M_1.$

%A warped product $M_1\times_{f}M_2$ is called: (i) proper, if the warped function $f$ is a non-constant function; (ii) convex, if the Hessian $\rm{H}^{f}$ is positive semi-definite at each point.

Within the context of this paper, our interst is in the warped product manifold $$M=\R_{>0}\times_{r}\frakH,$$ that is, the Riemannian cone on the Heisenberg group $\frakH$. It follows from $\Delta f=0$ that $\R_{>0}\times_{r}\frakH$ does not admit any isometric minimal immersion into any complex hyperbolic space. However, the manifold $M$ can be mapped to the complex hyperbolic plane by a horospherical map, see Definition \ref{def:horo}. In order to reveal more relations between these two Riemannian manifolds, it is natural to consider the metric of $M$ obtained from the Heisenberg group, see Section \ref{sec-rapprox} below.

\section{Heisenberg group}\label{sec-heis}
\subsection{Definition, contact structure}\label{sec:h}
The Heisenberg group $\frakH$  is the set $\C\times\R$ with multiplication $*$ given by
$$
(z,t)*(w,s)=(z+w,t+s+2\Im(z\overline{w})).
$$
The Heisenberg group $\frakH$ is a 2-step nilpotent Lie group. Consider the left-invariant vector fields
\begin{eqnarray*}
X=\frac{\partial}{\partial x}+2y\frac{\partial}{\partial t},\quad Y=\frac{\partial}{\partial y}-2x\frac{\partial}{\partial t},
\quad T=\frac{\partial}{\partial t}.
\end{eqnarray*}
We also use complex fields
$$
Z=\frac{1}{2}(X-iY)=\frac{\partial}{\partial z}+i\overline{z}\frac{\partial}{\partial t},\quad \overline{Z}=\frac{1}{2}(X+iY)=\frac{\partial}{\partial \overline{z}}-iz\frac{\partial}{\partial t}.
$$
The vector fields $X,Y,T$ form a basis for the Lie algebra of left-invariant vector fields of $\frakH$.
The Lie algebra $\mathfrak{h}$ of $\frakH$ has a grading $\mathfrak{h} = \mathfrak{v}_1\oplus \mathfrak{v}_2$ with
\begin{displaymath}
\mathfrak{v}_1 = \mathrm{span}_{\R}\{X, Y\}\quad \text{and}\quad \mathfrak{v}_2=\mathrm{span}_{\R}\{T\}.
\end{displaymath} 

In $\frakH$ we consider the 1-form 
\begin{equation}\label{eq-omega}
\omega=dt+2xdy-2ydx=dt+2\Im(\overline{z}dz).
\end{equation}
The following proposition holds; it summarises well-known facts about $\frakH$:

\medskip

\begin{prop}\label{prop:heis-specifics}
Let the Heisenberg group $\frakH$ together with  the 1-form $\omega$ as in (\ref{eq-omega}). Then the
manifold $(\frakH,\omega)$ is pseudo-hermitian. Explicitly:
 \begin{enumerate}
  \item The form $\omega$ of $\frakH$ is left-invariant.
  \item If $dm$ is the Haar measure for $\frakH$ then $dm=-(1/4)\;\omega \wedge d\omega$.
  \item The kernel of $\omega$ is generated by 
$X$ and $Y$.
  \item The Reeb vector field for $\omega$ is $T$.
  \item The only non trivial Lie bracket relation is
  $
  [X,Y]=-4T.
  $
  \item Let $\calH=\ker(\omega)$ and consider the almost complex structure $J$ defined on $\calH$ by
 $JX=Y$, $ JY=-X.$
 Then  $J$ is compatible with $d\omega$ and moreover, $\calH$ is a strictly pseudoconvex {\rm CR} structure; that is, $d\omega$ is positively oriented on $\calH$. 
 \end{enumerate}
\end{prop}
The sub-Riemannian structure of $\calH$ is defined by the relations
$$
g_{cc}(X,X)=g_{cc}(Y,Y)=1,\quad g_{cc}(X,Y)=0.
$$
The sub-Riemannian metric is then given by
$$
g_{cc}=ds_{cc}^2=dx^2+dy^2.
$$
The isometry group ${\rm Isom}_{cc}(\frakH)$  of the sub-Riemannian  metric $g_{cc}$ comprises compositions of: 
\begin{enumerate}
\item Left-translations $T_{(\zeta,s)}$, $(\zeta,s)\in\frakH$,  defined by
 $
 T_{(\zeta,s)}(z,t)=(\zeta,s)*(z,t). 
 $ 
 The group of left-translations is isomorphic to $\frakH$.
 \item  Conjugation $j$, defined by
 $
 j(z,t)=(\overline{z},-t).
 $
\item  Rotations $R_\theta$, $\theta\in\R$, defined by 
 $
 R_\theta(z,t)=(ze^{i\theta},t)
 $
 for every $(z,t)\in\frakH$. The group of rotations is isomorphic to ${\rm U}(1)$.
 \end{enumerate}
 Left-translations and rotations are {\rm CR} maps which preserve $\omega$ whereas conjugation is anti-{\rm CR} which skew-preserves $\omega$: $j^*(\omega)=-\omega$. 
The isometry group of $g_{cc}$ comprises of composites of the above mappings:
$$
{\rm Isom}(\frakH, g_{cc})\simeq\frakH\times {\rm U}(1)\times\mathbb{Z}_2.
$$
The dilations $ D_\delta$ ($\delta>0$) which are defined by
 $$
 D_\delta(z,t)=(\delta z, \delta^2 t)
 $$
for every $(z,t)\in\frakH$ are homotheties for the metric $g_{cc}$ and they are also $\rm{CR}$ maps. The group of dilations is isomorphic to the multiplicative group $\R_{>0}$.

\medskip
%We then call $(M;\eta,\xi,\phi,g)$ the contact Riemannian structure  on $M$ associated to the pseudo-hermitian structure $(M,\eta)$.
\subsection{From Riemannian approximants to contact Riemannian structure}\label{sec-rapprox}
Using the sub-Riemannian metric of $\frakH$ we construct contact Riemannian structure in the Heisenberg group as follows: 
for $L>0$ we consider a Riemannian metric in $\frakH$ such that the frame 
$$
\{X,Y, T_L=T/\sqrt{L}\}
$$ 
is orthonormal. Explicitly,
$$
g_L=ds_L^2=dx^2+dy^2+L\omega^2=ds^2_{cc}+L\omega^2.
$$
Note here that $g_L$ is the metric defined in \cite{CDPT} (mind only the different notation of the definition of the Heisenberg group).
Let $\phi$ be the extension of $J$ in the whole tangent bundle ${\rm T}(\frakH)$ by setting 
$\phi (T_L)=0.$
Then $(\frakH;\sqrt{L}\omega,T_L,$ $ \phi,g_L)$ is contact Riemannian structure if and only if $L=1/4$ by checking the equations (\ref{eq:contactmetric}): the equation
\[
(1/2)\,d(\sqrt{L}\omega)(X, Y)=g_L(\phi X, Y),
\]
is equivalent to $$(\sqrt{L}/2)\,d\omega (X, Y)=2\sqrt{L}=g_L(Y, Y)=1,$$ i.e., $L=1/4.$
From now, we will write $g$ instead of $g_{1/4}$, $\widetilde{T}$ instead of $2T$ and $\widetilde{\omega}$ instead of $(1/2)\omega$.

\medskip

Let $\nabla$ be the Levi-Civita connection. Using Koszul's formula (in the case of orthonormal frames)
\begin{equation}\label{eq-Koszul}
g(\nabla_VU,W)=-(1/2)\left(g([U,W],V)+g([V,W],U)+g([U,V],W)\right),
\end{equation}
we find
\begin{eqnarray*}
&&
\nabla_{X}X=0,\quad \nabla_{X}Y=-\widetilde{T},\quad \nabla_{X}\widetilde{T}=Y,\\
&&
\nabla_{Y}X=\widetilde{T},\quad\nabla_{Y}Y=0,\quad \nabla_{Y}\widetilde{T}=-X,\\
&&
\nabla_{\widetilde{T}}X=Y,\quad \nabla_{\widetilde{T}}Y=-X,\quad \nabla_{\widetilde{T}}\widetilde{T}=0.
\end{eqnarray*}
Let
\begin{equation}\label{eq-curvten}
R(U,V)W=\nabla_{V}\nabla_{U}W-\nabla_{U}\nabla_{V}W
+\nabla_{[U,V]}W,
\end{equation}
be the Riemannian curvature tensor. We have
\begin{eqnarray*}
&&
R(X,\widetilde{T})X=\widetilde{T},\quad R(X,\widetilde{T})Y=0,\quad R(X,\widetilde{T})\widetilde{T}=-X,\\
&&
R(Y,\widetilde{T})X=0,\quad R(Y,\widetilde{T})Y=\widetilde{T},\quad R(Y,\widetilde{T})\widetilde{T}=-Y\\
&&
R(X,Y)X=-3Y,\quad R(X,Y)Y=3X,\quad R(X,Y)\widetilde{T}=0.
\end{eqnarray*}
Using the relation
$
K(U,V)=g(R(U,V)U,V)
$
for sectional curvature of planes spanned by unit vectors $U,V$
we obtain the following:
\begin{cor}\label{cor:secH}
The sectional curvatures of the distinguished planes spanned by a) $X,Y$, b) $X,\widetilde{T}$ and c) $Y,\widetilde{T}$ are, respectively: 
\begin{eqnarray*}
&&
K(X,Y)=
-3,\quad
K(X,\widetilde{T})=
1,\quad
K(Y,\widetilde{T})=
1.
\end{eqnarray*}
\end{cor}

\subsection{Sasakian structure}
A contact Riemannian manifold $(M;\eta,\xi,\phi,g_M)$ such that:
\begin{enumerate}
\item The Reeb vector field $\xi$ is Killing (equivalently, $\xi$ is an infinitesimal {\rm CR} transformation);
\item $\xi$ is unit vector field and
$$
R(X,\xi)Y=g_M(X,Y)\xi-g_M(\xi,Y)X,
$$
for all vector fields $X,Y$ in ${\mathfrak X}(M)$,
\end{enumerate} 
is called {\it Sasakian}. In our case, we have first that the Reeb vector field  $\widetilde{T}$ is by definition unit; it is also Killing.
\begin{lem}\label{lem-kill}
The Reeb vector field $\widetilde{T}$ is Killing for the metric $g.$
\end{lem}
\begin{proof}
It suffices to show that for every vector fields $U,V$ we have
$$
g(\nabla_V\widetilde{T},U)+g(\nabla_U\widetilde{T}, V)=0.
$$
Set $U=a_1X+b_1Y+c_1\widetilde{T}$, $V=a_2X+b_2Y+c_2\widetilde{T}$. Then
\begin{eqnarray*}
&&
\nabla_U\widetilde{T}=a_1Y-b_1X,\\
&&
\nabla_V\widetilde{T}=a_2Y-b_2X
\end{eqnarray*}
and
$$
g(\nabla_V\widetilde{T},U)+g(\nabla_U\widetilde{T}, V)=a_2b_1-b_2a_1+a_1b_2-b_1a_2=0.
$$
\end{proof}
Now, if $U$ and $V$ are as above then
$$
R(U,\widetilde{T})V=-a_1c_2X-b_1c_2Y+(a_1a_2+b_1b_2)\widetilde{T}.
$$
On the other hand, a direct calculation yields to
$$
g(U,V)\widetilde{T}-g(\widetilde{T},V)U=-a_1c_2X-b_1c_2Y+(a_1a_2+b_1b_2)\widetilde{T}.
$$
From Lemma \ref{lem-kill} and the above relations we have
\begin{prop}\label{pro-Sas}
The structure $(\widetilde{\omega},\widetilde{T},$ $ \phi,g)$ on $\frakH$ is Sasakian. 
\end{prop}\label{prop:HSas}

\subsection{$\rm{CR}$ and Sasakian automorphisms}
\label{sec-isomsas}

Let $\mathfrak{CR}(\frakH)$ be the group of $\rm{CR}$ maps of $\frakH$ and let also $\mathfrak{Aut}(\frakH)$ be the group of Sasakian automorphisms of $\frakH$: that is, $g$-isometries $f$ which also satisfy $f^*\widetilde{\omega}=\widetilde{\omega}$. Denote by $(\mathfrak{CR}(\frakH))_0$ and $(\mathfrak{Aut}(\frakH))_0$ their respective connected components of identity. We have the following theorem (Theorem 5.3 of \cite{Bo}):
\begin{thm}
The group  $(\mathfrak{CR}(\frakH))_0$ is isomorphic to the semi-direct product $({\rm U}(1)\times \R_{>0})\ltimes\frakH$. The group $(\mathfrak{Aut}(\frakH))_0$ is isomorphic to the semi-direct product ${\rm U}(1)\ltimes\frakH$.
\end{thm}

\section{K\"ahler forms on $\frakH\times\R_{>0}$}%\calC(\frakH)$}
\label{sec-Kahler}
\subsection{K\"ahler form I}
 Let $\calC(M)$ be the cone over a $n$-dimensional Riemannian manifold $(M, g)$. The  manifold $(M, g)$ is Sasakian if and only if the holonomy group of $\calC(M)$ reduces to a subgroup of unitary group. Thus $\calC(M)=(M\times \R_{>0},dr^2+r^2g)$ is K\"ahler with dimension $n+1.$ Now Proposition \ref{pro-Sas} immediately implies that
$
\calC(\frakH)=(\frakH\times\R_{>0},\J,g_r,\Omega_r)
$
is K\"ahler. We describe below the features of $\calC(\frakH)$: first, we consider the orthonormal basis $\{X_r,Y_r,T_r,\partial_r\}$  for the metric $g_r$, that is,
$$
X_r=(1/r)X,\quad Y_r=(1/r)Y,\quad T_r=(1/r)\widetilde{T},\quad \partial_r=d/dr.
$$
We note that all Lie brackets vanish besides
$$
[X_r,Y_r]=-(2/r)T_r,\quad[X_r,\partial_r]=(1/r)X_r,
\quad[Y_r,\partial_r]=(1/r)Y_r,\quad[T_r,\partial_r]=(1/r)T_r.
$$
The action of the complex structure $\J$ is given by
$$
\J X_r=Y_r,\quad \J Y_r=-X_r,\quad \J T_r=-\partial_r,\quad \J \partial_r=T_r.
$$
Let also
$
\phi=dx$ and $ \psi=dy.
$
The following hold:
$$
r\phi=X_r^*,\quad r\psi=Y_r^*,\quad r\widetilde{\omega}=T_r^*,\quad dr=\partial_r^*.
$$
We write
$\phi_1=r(\phi+i\psi)$, $\phi_2=dr+ir\widetilde{\omega}$ and $Z_r=(1/2)(X_r-iY_r)$, $V_r=(1/2)(\partial_r-iT_r)$, so that $\phi_1=Z_r^*$ and $\phi_2=V_r^*$.
The K\"ahler metric $g_r$ and the K\"ahler form $\Omega_r$ are given respectively by
\begin{eqnarray*}
&&\label{eq:grHC}
g_r=dr^2+r^2g=dr^2+r^2(\phi^2+\psi^2+(\widetilde\omega)^2)=|\phi_1|^2+|\phi_2|^2,\\
&&\label{eq:Kr}
\Omega_r=d\left(\frac{r^2}{2}\widetilde\omega\right)=rdr\wedge\widetilde\omega+r^2\;\phi\wedge\psi=\frac{i}{2}(\phi_1\wedge \overline{\phi_1}+\phi_2\wedge \overline{\phi_2}).
\end{eqnarray*}
%We also note that by setting $\phi_1=r(\phi+i\psi)$ and $\phi_2=dr+ir\omega'$ we obtain the complex expressions for $g_r$ and $\Omega_r$:
%\begin{equation}\label{eq:KH}
%g_r=|\phi_1|^2+|\phi_2|^2,\quad\Omega_r=\frac{i}{2}(\phi_1\wedge \overline{\phi_1}+\phi_2\wedge \overline{\phi_2}).
%\end{equation}
\subsubsection{Curvature}
If $\nabla^r$ is the Riemannian connection, we obtain by Koszul's formula that
\begin{eqnarray*}
&&
\nabla^r_{X_r}X_r=-(1/r)\partial_r,\quad \nabla^r_{Y_r}X_r=(1/r)T_r,\quad \nabla^r_{T_r}X_r=(1/r)Y_r,\quad\nabla^r_{\partial_r}X_r=0,\\
&&
\nabla^r_{X_r}Y_r=-(1/r)T_r,\quad \nabla^r_{Y_r}Y_r=-(1/r)\partial_r,\quad \nabla^r_{T_r}Y_r=-(1/r)X_r,\quad\nabla^r_{\partial_r}Y_r=0,\\
&&
\nabla^r_{X_r}T_r=(1/r)Y_r,\quad \nabla^r_{Y_r}T_r=-(1/r)X_r,\quad \nabla^r_{T_r}T_r=-(1/r)\partial_r,\quad\nabla^r_{\partial_r}T_r=0,\\
&&
\nabla^r_{X_r}\partial_r=(1/r)X_r,\quad\nabla^r_{Y_r}\partial_r=(1/r)Y_r,\quad\nabla^r_{T_r}\partial_r=(1/r)T_r,\quad\nabla^r_{\partial_r}\partial_r=0.
\end{eqnarray*}

In the next proposition we compute the sectional curvatures of distinguished planes. 
\begin{prop}\label{prop-sec-curve}
The sectional curvatures at all other pairs of distinguished planes vanish besides that of the distinguished plane spanned by $X_r,Y_r$:
$$
K_r(X_r,Y_r)=-4/r^2<0.
$$
\end{prop}
\begin{proof}
If $R_r$  is the Riemannian curvature tensor, we have 
\begin{eqnarray*}
&&
R_r(X_r,Y_r)X_r=-(4/r^2)Y_r,
\end{eqnarray*}
and
\begin{eqnarray*}
&&
R_r(X_r,T_r)X_r=
R_r(X_r,\partial_r)X_r=
R_r(Y_r,T_r)Y_r\\
&&
=
R_r(Y_r,\partial_r)Y_r=
R_r(T_r,\partial_r)T_r=0.
\end{eqnarray*}
Thus the holomorphic sectional curvature of the plane spanned by $X_r$ and $Y_r$ is
$
K_r(X_r,Y_r)=-4/r^2<0.
$ 
% and the holomorphic sectional curvature of the planes spanned by $T_r$ and $\partial_r$ is $K_r(T_r,\partial_r)=-1/r^2<0.$ 
All other sectional curvatures of distinguished planes vanish.
\end{proof}
\begin{cor}\label{cor-cur}
The Ricci curvatures of $g_r$ in the directions of $X_r, Y_r, T_r$ and $\partial_r$ are respectively
\begin{eqnarray*}
&&
{\rm Ric}(X_r)={\rm Ric}(Y_r)=-\frac{4}{3r^2},\quad
{\rm Ric}(T_r)={\rm Ric}(\partial_r)=0,
\end{eqnarray*} 
and the scalar curvature is
\begin{equation*}
K=-\frac{2}{3r^2}.
\end{equation*}
\end{cor}
\subsubsection{Submanifolds}\label{sec-submanifolds}

\medskip

{\bf (i) Upper half-plane.} We embed $\calU=\{(t,r)\;:t\in\R, r>0\}$ into $\calC(\frakH)$ by setting
$$
\iota_{\calU}(t,r)=(0,2t,r).
$$
The pullback metric is then
$$
g_\calU=\iota_{\calU}^{\ast}g_r=dr^2+r^2dt^2
$$
and is a flat K\"ahler metric on $\calU$ and the submanifold $\calU$ is a totally geodesic submanifold of $\calC(\frakH)$ (the second fundamental form vanishes). To see this, let $\{\partial_r,\;(1/r)\partial_t\}$ be an (orthonormal) basis for $\calU$ and let $\{\partial_r,\; (1/2)T_r\} $ be a local extension to the tangent bundle of $\calC(\frakH)$. If $B$ is the second fundamental form then
\begin{eqnarray*}
&&
B(\partial_r,\partial_r)=\left(\nabla^r_{\partial_r}\partial_r\right)^N=0^N=0,\\
&&
B(\partial_r,(1/r)\partial_t)=\left((1/2)\nabla^r_{\partial_r}T_r\right)^N=0^N=0,\\
&&
B((1/r)\partial_t,(1/r)\partial_t)=\left((1/4)\nabla^r_{T_r}T_r\right)^N=\left(-1/(4r)\partial_r\right)^N=0.
\end{eqnarray*}
Therefore we obtain that the second fundamental form is identically zero and hence $(\calU,g)$ is totally geodesic. As for the curvature, it follows from Proposition \ref{prop-sec-curve}, Gauss Theorem and the vanishing of the second fundamental form that the sectional  curvature is zero.

%We note that another embedding of $\calU$ can be given by
%$$
%(x,r)\mapsto(x,0,0,r)
%$$
%and the metric here is $dr^2+r^2dx^2$.
%It is straightforward to show that this embedding is not totally geodesic (the second fundamental for is not zero) but the curvature is again zero. In fact, we have:
%\begin{eqnarray*}
%&&
%B(\partial_x,\partial_x)=\left(\nabla^r_{\partial_x}\partial_x\right)^N=-2yx\partial_t,\\
%&&
%B(\partial_r,\partial_x)=\left(\nabla^r_{\partial_r}\partial_x\right)^N=0^N=0,\\
%&&
%B(\partial_r,\partial_r)\left(\nabla^r_{\partial_r}\partial_r\right)^N=0^N=0.
%\end{eqnarray*}
%Therefore, by Gauss equation we obtain that $K(\partial_x,\partial_r)=0$. 

\medskip

{\bf (ii) Complex plane.} We embed $\C$ into $\calC(\frakH)$ by setting
$$
\iota_{\C}(z)=(z,0,1).
$$
The pullback metric is 
\begin{equation}\label{metric-C}
g_\C=\iota_{\C}^{\ast}g_r=|dz|^2+\Im^2({\overline z}dz)=(1+y^2)dx^2-2xydxdy+(1+x^2)dy^2,
\end{equation}
which is again  K\"ahler.  It is not totally geodesic: to see this, we consider vector fields $\partial_x$ and $\partial_y$ on $\C$ and their respective local extensions to $\calC(\frakH)$. The normal space to $\C$ is spanned by the orthonormal vector fields
$$
N_1=\frac{y\partial_x-x\partial_y+2(1+x^2+y^2)\partial_t}{\sqrt{1+x^2+y^2}},\quad N_2=\partial_r.
$$
 If $B$ is the second fundamental form, then
\begin{eqnarray*}
&&
B(\partial_x,\partial_x)=\frac{2xy}{\sqrt{1+x^2+y^2}} N_1-(1+y^2)N_2,\\
&&
B(\partial_x,\partial_y)=\frac{y^2-x^2}{\sqrt{1+x^2+y^2}}N_1+xyN_2,\\
&&
B(\partial_y,\partial_y)=-\frac{2xy}{\sqrt{1+x^2+y^2}}N_1-(1+x^2)N_2.
\end{eqnarray*}
Thus, submanifold $\C$ is not a  totally geodesic submanifold of $\calC(\frakH)$.
Now,
$
R_r(\partial_x,\partial_y)\partial_x=R_r(X,Y)X=-4\partial_y+8x\partial_t
$
and $g_r(R_r(\partial_x,\partial_y)\partial_x,\;\partial_y)=-4$ and by Gauss equation we obtain
%Since
%$$
%g^r(B(\partial_x,\partial_y),B(\partial_x,\partial_y))=
%$$
$$
K(\partial_x,\partial_y)=-\frac{3+2x^2+2y^2}{(1+x^2+y^2)^2}.
$$
In this way, $\C$ is a submanifold with unbounded negative (holomorphic) sectional curvature. 

\medskip

{\bf (iii) Upper half-space.} We embed the upper half space $\calU_3=\C\times\R_{>0}$ into $\calC(\frakH)$ by setting
$$
\iota_{\calU_3}(z,r)=(z,0,r)
$$
and the induced metric is
$$
g_{\calU_3}=dr^2+r^2\left((1+y^2)dx^2-2xydxdy+(1+x^2)dy^2\right);
$$
thus $\calU_3$ is indeed the warped product $\R_{>0}\times _r\C$ and $\C$ here is considered with the metric as in (\ref{metric-C}).
The tangent space of $\calU_3$ is generated by the vector fields $(1/r) \partial_x$, $(1/r)\partial_y$, $\partial r$ which are all normal to the unit vector field
$$
N=\frac{y\partial_x-x\partial_y+2(1+x^2+y^2)\partial_t}{r\sqrt{1+x^2+y^2}}.
$$
 If $B$ is the second fundamental form, then simple calculations deduce
\begin{eqnarray*}
&&
B((1/r)\partial_x,(1/r)\partial_x)=\frac{2xy}{r\sqrt{1+x^2+y^2}}N,\quad B((1/r)\partial_x,(1/r)\partial_y)=\frac{y^2-x^2}{r\sqrt{1+x^2+y^2}}N\\
&&
B((1/r)\partial_y,(1/r)\partial_y)=-\frac{2xy}{r\sqrt{1+x^2+y^2}}N,\quad B((1/r)\partial_x,\partial_r)=0,\\
&&
B((1/r)\partial_y,\partial_r)=0,\quad B(\partial_r,\partial_r)=0.
\end{eqnarray*}
  Therefore $\calU_3$ is not totally geodesic. By Gauss equation we obtain for the sectional curvatures of distinguished planes that 
  $$
  K((1/r)\partial_x,(1/r)\partial_y)=-\frac{4+(x^2+y^2)^2}{r^2(1+x^2+y^2)^2}
  $$
  whereas all other sectional curvatures vanish.
 \medskip

{\bf (iv) Heisenberg group.} We embed the Heisenberg group $\frakH$ into $\calC(\frakH)$ by setting
$$
\iota_{\frakH}(z,t)=(z,t,1)
$$ 
and the induced metric is of course $g$ as in Section \ref{sec-rapprox}.
The tangent space of $\frakH$ is generated by the vector fields $X,Y,\widetilde{T}$, which are all normal to the unit field
$
N=\partial_r.
$
 If $B$ is the second fundamental form, then 
\begin{eqnarray*}
&&
B(X,X)=B(Y,Y)=B(\widetilde T,\widetilde T)=-N,
\end{eqnarray*}
whereas all other components of $B$ vanish.
  Therefore $\frakH$ is not totally geodesic. Now by Gauss equation we subsequently recover the formulas for the sectional curvatures of distinguished planes as in Corollary \ref{cor:secH}.

\subsection{K\"ahler form II}
On $\frakH\times\R_{>0}$ we consider the basis of the tangent space comprising
$$
X,Y,T,\partial_r
$$ 
and we define an almost complex structure $\I$ on $\calC(\frakH)$ by the relations
$$
\I X=Y,\quad \I Y=-X,\quad \I T=\partial_r,\quad \I\partial_r=-T.
$$
For positive functions $a=a(r)$ and $b=b(r)$ we consider the Riemannian metrics $g_{a.b}$ on $\calC(\frakH)$ defined by
$$
g_{a,b}=\frac{dx^2+dy^2}{a^2}+\frac{\omega^2+dr^2}{b^2}.
$$
An orthonormal basis for this metric comprises the vector fields 
$$
X'=aX,\quad Y'=aY \quad T'=bT,\quad R'=b\partial_r,
$$
which satisfy the following bracket relations:
\begin{eqnarray*}
&&
[X',Y']=-(4a^2/b)T',\quad [X',T']=0,\quad [X',R']=-(b\dot a/a)X',\\
&&
[Y',T']=0,\quad [Y',R']=-(b\dot a/a)Y',\\
&&
[T',R']=-\dot b T'.
\end{eqnarray*}
The corresponding fundamental form for $g_{a,b}$ is
$$
\Omega_{a,b}=\frac{dx\wedge dy}{a^2}+\frac{\omega\wedge dr}{b^2}.
$$
The triple $(\I, g_{a,b}, \Omega_{a,b})$ is a Hermitian structure on $\frakH\times\R_{>0}$: we verify that $\I$ is integrable since the Nijenhuis tensor vanishes. Moreover,
\[
g_{a, b}(\I U, \I V)=g_{a, b}(U, V),\quad \Omega_{a, b}(U, V)=g_{a, b}(\I U, V)
\]
for any vector fields $U, V.$ 

\begin{lem}\label{lem-oabclosed}
The Hermitian manifold $(\frakH\times\R_{>0},\I, g_{a,b}, \Omega_{a,b})$ is K\"ahler if and only if $\dot a=2a^3/b^2$.
\end{lem}
\begin{proof}
The Hermitian manifold $(\frakH\times\R_{>0},\I, g_{a,b}, \Omega_{a,b})$ is K\"ahler if and only if $\Omega_{a,b}$ is closed. We check at once that
\begin{eqnarray*}
d\Omega_{a,b}&=&-2(\dot a/a^3)dr\wedge dx\wedge dy+(4/b^2)dx\wedge dy\wedge dr\\
&=&\left(4/b^2-2\dot a/a^3\right)dx \wedge dy\wedge dr.
\end{eqnarray*}
%The Hermitian manifold $(\calC(\frakH),\I, g_{a,b}, \Omega_{a,b})$ is K\"ahler if and only if $\Omega_{a,b}$ is closed, i.e., $\dot a=2a^3/b^2.$
\end{proof}
If $\nabla$ is the Riemannian connection, then
\begin{eqnarray*}
&&
\nabla_{X'}X'=(b\dot a/a)R',\quad \nabla_{Y'}X'=(2a^2/b)T',\quad \nabla_{T'}X'=(2a^2/b)Y',\quad\nabla_{R'}X'=0,\\
&&
\nabla_{X'}Y'=-(2a^2/b)T',\quad \nabla_{Y'}Y'=(b\dot a/a)R',\quad \nabla_{T'}Y'=-(2a^2/b)X',\quad\nabla_{R'}Y'=0,\\
&&
\nabla_{X'}T'=(2a^2/b)Y', \quad \nabla_{Y'}T'=-(2a^2/b)X',\quad \nabla_{T'}T'=\dot bR',\quad\nabla_{R'}T'=0,\\
&&
\nabla_{X'}R'=-(b\dot a/a)X',\quad\nabla_{Y'}R'=-(b\dot a/a)Y',\quad\nabla_{T'}R'=-\dot bT',\quad\nabla_{R'}R'=0.
\end{eqnarray*}
By requiring $\Omega_{a,b}$ to be closed, we obtain that
\begin{eqnarray*}
&&
\nabla_{X'}X'=\nabla_{Y'}Y'=(2a^2/b)R',\\
&&
\nabla_{X'}R'=-(2a^2/b)X',\quad \nabla_{Y'}R'=-(2a^2/b)Y'.
\end{eqnarray*}
By calculating the Riemannian curvature tensor $R,$ we find that
\begin{eqnarray*}
&&
R(X',Y')X'=-(16a^4/b^2)Y',\quad R(T',R')T'=(b\ddot b-\dot b^2)R'\\
&&
R(X',T')X'=\left((2a^2/b)^2-(2a^2/b)\dot b\right)T',\quad R(X',R')X'=\left((d/dr)(2a^2/b)-(2a^2/b)\right)R',\\
&&
R(Y',T')Y'=\left((2a^2/b)^2-(2a^2/b)\dot b\right)T',\quad R(Y',R')Y'=\left((d/dr)(2a^2/b)-(2a^2/b)\right)R'
\end{eqnarray*}
Therefore, it follows that
\begin{eqnarray*}
&&
K(X',Y')=-16a^4/b^2,\quad K(T',R')=b\ddot b-\dot b^2,\\
&&
K(X',T')=(2a^2/b)^2-(2a^2/b)\dot b,\quad K(X',R')=(d/dr)(2a^2/b)-(2a^2/b),\\
&&
K(Y',T')=(2a^2/b)^2-(2a^2/b)\dot b,\quad K(Y',R')=(d/dr)(2a^2/b)-(2a^2/b).
\end{eqnarray*}
At this point we need the following lemma which is verified after straightforward calculations:
\begin{lem}
If $C>0$, equations
$$
-16a^4/b^2=-C^2, \quad b\ddot b-\dot b^2=-C^2, 
$$
have positive solutions 
$$
a=\sqrt{Cb}/2, \quad b(r)=Cr. %, \quad b(r)=(C/c)\sinh(\pm cr+d),\quad c>0, d\in\R.
$$
For such $a$ and $b$, we have:
\begin{eqnarray*}
&&
 K(X',Y')=K(T',R')=-C,\\
 &&
 K(X',R')=K(Y',R')=-C/2,\\
 &&
K(X',T')=K(Y',T')=C^2/4-C/2.
\end{eqnarray*}
\end{lem} 
Observe that the sectional curvature is negative if $0<C<2$.
The following proposition now follows immediately.
\begin{prop} \label{pro-hol}
The set $(\frakH\times\R_{>0},\I, g_{a,b}, \Omega_{a,b})$ is a K\"ahler manifold of constant holomorphic sectional curvature $-1$ if and only if $a=\sqrt{r}/2$ and $b=r$. For those values of $a$ and $b$, the real sectional curvature is pinched between $-1$ and $-1/4$.
\end{prop}
For $a$ and $b$ as above we shall write $g'$ and $\Omega'$ instead of $g_{a,b}$ and $\Omega_{a,b}$, respectively. 
 
 \section{{\rm PCR}-equivalence}\label{sec-PCR}
 In this section we review {\rm PCR}-mappings in Section \ref{sec-PCR-1} and we define {\rm PCR}-K\"ahler equivalence in Section \ref{sec:Kahler} and prove Theorem \ref{thm-PCRK1}.
\subsection{{\rm PCR}-mappings}\label{sec-PCR-1}
The definitions stated below are based on definitions given in \cite{Pla-CR} as well as in \cite{Be}.
\begin{defi}\label{def:PCR}
Let $M$ be a manifold with a {\rm CR} structure $(\calH,J)$. Let $(N,\J)$ be a complex manifold. Suppose that $\calH'\subset{\rm T}^{(1,0)}(N)$ and there exists an immersion $F:M\to N$ such that $F_*(\calH)=\calH'_{F(M)}$. The map $F$ is called $(\calH,\calH')$-{\rm PCR}.
\end{defi}
We comment that the condition of the definition implies $F_*(JX)=\J(F_*X)$ for every vector field $X$ in the underlying real subbundle of $\calH$. We also wish to comment that this is a weaker version of the definition of {\rm CR} submanifold, as it is given  by Bejancu in \cite{Be}, p. 20.

\medskip 

\begin{defi}\label{def:PCRim}
Let $M$, $N$, $\calH,$ $\calH'$ and  $F:M\to N$ an $(\calH,\calH')$-{\rm PCR} immersion be as in Definition \ref{def:PCR}. Suppose now that $\calH$ is strictly pseudoconvex with contact form $\eta$, Reeb vector field $\xi$ and contact metric $g=(1/2)d\eta$. Suppose also that $N$ is a K\"ahler manifold with metric $g$ and fundamental form $\Omega$. Then an immersion $F:M\to N$ is called $(\calH,\calH')$-K\"ahler if $F^*(g_{\mid\calH'\times\calH'})=g_{\mid\calH\times\calH}$. 
\end{defi}

Again, compare Definition \ref{def:PCRim} with Bejancu's definition. Note that our definition implies $F_*\xi\in(\calH')_{F(M)}^\perp$. We are interested in the particular case which is described in the next proposition.

\medskip

\begin{prop}\label{prop:HHK}
Let $M=(M;\eta,\xi,\phi,g)$ be a Sasakian manifold, $\calH=\ker(\eta)$, and let also $N=(\calC(M);\J,g_r,\Omega_r)$ be its K\"ahler cone and $\calH_r=\ker(dr+ir\eta)\subset{\rm T}^{(1,0)}(\calC(M))$. Then the embedding $\iota:M\to\calC(M)$ as the hypersurface $r=1$ is $(\calH,\calH_r)$-K\"ahler. 
\end{prop}
\begin{proof}
Let $\calH=\ker(\eta)$ be the strictly pseudoconvex CR structure of $M$, and let $\iota:M\to\calC(M)$ be the mapping $p\mapsto(p,1)$. Then $\iota$ is clearly an embedding and if $Z_p\in\calH_p$, $p\in M$, then
$$
\iota_{*,p}(Z_p)=Z_{(p,1)}(=Z_p+{\bf 0}_1)\in(\calH_r)_{(p,1)}.
$$ 
Hence $\iota$ is $(\calH,\calH_r)$-{\rm PCR}. On the other hand, $(g_r)_{|\calH_r\times\calH_r}=r^2g_{cc}$ and thus
$$
\iota^{*}((g_r)_{|\calH_r\times\calH_r})=g_{cc}=g_{\mid\calH\times\calH}.
$$
%Finally, $$\iota_{*,p}(\xi_p)=\xi_{(p,1)}\in(\calH_r)_{(p,1)}^\perp.$$
\end{proof}
From Example iv) of Section \ref{sec-submanifolds} and Proposition \ref{prop:HHK} we obtain the following:% counterpart of Proposition \ref{prop:HHK1}:
\begin{prop}\label{prop:HHHK}
The Heisenberg group $\frakH$ is embedded into $\calC(\frakH)$ as the hypersurface $r=1$. If $\calH$ is the {\rm CR} structure of $\frakH$ and $\calH_r=\ker(dr+ir\omega)$, then the embedding is $(\calH,\calH_r)$-K\"ahler.
\end{prop}

\subsection{{\rm PCR}-K\"ahler equivalence}\label{sec:Kahler}
\begin{defi}\label{def:PCRK}
 Let $(N,\J,g,\Omega)$ and $(N',g',\Omega',\J')$ be two K\"ahler manifolds of the same dimension and let $G:N\to N'$ be a diffeomorphism. Let also $\calH$ be a subbundle of ${\rm T}^{(1,0)}(N)$ and $\calH'$ be a subbundle of ${\rm T}^{(1,0)}(N')$. The map $G$ is called:
 \begin{itemize}
 \item[{i)}] {\it $(\calH,\calH')$-PCR} if  $G_*(\calH)=\calH'$;
 \item[{ii)}] {\it $(\calH,\calH')$-K\"ahler} if $G^*(g'_{\mid\calH'\times\calH'})=g_{\mid\calH\times\calH}$.
 \end{itemize}
 \end{defi}
 Note that condition ii) is equivalent to
\begin{itemize}
 \item[{ii')}] {\it $(\calH,\calH')$-K\"ahler} if $G^*(\Omega'_{\mid\calH'\times\calH'})=\Omega_{\mid\calH\times\calH}$.
 \end{itemize}
 When the above hold, the manifolds $N$ and $N'$ are called {\rm PCR}-K\"ahler equivalent. 
 \begin{thm}\label{thm-PCRK1}
 The manifolds $(\calC(\frakH),\J, g_r,\Omega_r)$ and $(\frakH\times\R_{>0},\I, g',\Omega')$ are {\rm PCR}-K\"ahler equivalent.
 \end{thm}
 \begin{proof}
 Let $G:(\calC(\frakH),\J, g_r,\Omega_r)\to(\frakH\times\R_{>0},\I, g',\Omega')$ with formula
 $$
 G(z,t,r)=(z,t,2\sqrt{r}).
 $$
 If $\calH=\langle X,Y\rangle$, then we have $G_*(Z)=Z$, hence $G$ is $(\calH,\calH)$-PCR. On the other hand,
 $$
 G^*g'{\mid_{\calH\times\calH}}=(4/r)(dx^2+dy^2)=g_r{\mid_{\calH\times\calH}}
 $$ 
 and thus $G$ is also $(\calH,\calH)$-K\"ahler.
 \end{proof}
 
 \section{Complex hyperbolic plane and $\calC(\frakH)$}\label{sec:chp}
 \subsection{Complex hyperbolic plane}\label{sec-chp}
The Heisenberg group $\frakH$ appears naturally within the context of complex hyperbolic geometry as the boundary of the complex hyperbolic plane. In the concept of this paper, the complex hyperbolic plane $\bH^2_\C$ is the one point compactification of the Siegel domain
$
\calS=\{(z_1,z_2)\in\C^2\;|\;\rho(z_1,z_2)>0\},
$
where
$
\rho(z_1,z_2)=-2\Re(z_1)-|z_2|^2.
$
The complex hyperbolic plane $\bH^2_\C$ is a complex manifold; there is a natural K\"ahler structure defined on $\bH^2_\C$ coming from the Bergman metric:
\begin{equation}\label{eq:gyp}
ds^2=\frac{4}{\rho}|dz_2|^2+\frac{4}{\rho^2}|\partial\rho|^2=\frac{4}{\rho}|dz_2|^2+\frac{4}{\rho^2}|dz_1+\overline{z_2}dz_2|^2.
\end{equation}
 The K\"ahler form is then
\begin{equation}\label{eq:Kchp1}
\Omega=-4i\partial\overline{\partial}(\log\rho)=-4i\left(\frac{1}{\rho}dz_2\wedge d\overline{z_2}-\frac{1}{\rho^2}\partial\rho\wedge\overline{\partial}\rho\right).
%2i\left(\frac{1}{\rho}dz_2\wedge d\overline{z_2}-\frac{1}{\rho^2}(dz_1+\overline{z_2}dz_2)\wedge(d\overline{z_1}+z_2 d\overline{z_2})
\end{equation} 
%Note that we may also write
%\begin{equation}\label{eq:Kchp}
%\Omega=d d^c (\log(\rho))=2\;d\left(\frac{\omega'}{\rho}\right),\quad\omega'=dy_1+x_2\;dy_2-y_2\;dx_2=\Im(dz_1+\overline{z_2}dz_2).
%\end{equation}
The group of  holomorphic isometries is ${\rm PU}(2,1)$. 
\subsection{Horospherical map}\label{sec-horo}
The boundary $\partial\calS$ of $\calS$ ($\rho(z_1,z_2)=0)$ admits a strictly pseudoconvex {\rm CR} structure with contact form $\omega'=-\Im(\partial\rho)$, and with this {\rm CR} structure $\partial\calS$ and the Heisenberg group $\frakH$ are {\rm CR} equivalent; the {\rm CR} diffeomorphism between them is given by
\begin{equation}\label{eq:h}
h: \frakH\ni(z, t)\longmapsto\left((-|z|^2+it)/2,z\right)\in\partial\calS,
\end{equation}
which identifies $\frakH$ to $\partial\bH^2_\C$ in a {\rm CR} manner. To see this, we calculate
\begin{eqnarray*}
h_*(Z)&=&(1/2)Z(-|z|^2+it)(\partial/\partial z_1)+(1/2)Z(-|z|^2-it)(\partial/\partial \overline{z_1})\\
&&+Z(z)(\partial/\partial z_2)+Z(\overline{z})(\partial/\partial \overline{z_2})\\
&=&-\overline{z_2}(\partial/\partial z_1)+(\partial/\partial z_2)\in{\rm T}^{(1,0)}(\calS).
\end{eqnarray*}
 The map $h$ is the boundary map of the {\it horospherical map} which describes the {\it horospherical model} for complex hyperbolic plane:
\begin{defi}\label{def:horo}
The horospherical model for $\bH^2_\C$ is given by the horospherical map  defined by
\begin{equation*}
H: \frakH\times\R_{>0}\ni(z, t,r)\longmapsto\left((-|z|^2-r+it)/2,\;z\right)\in\calS.
\end{equation*}
\end{defi}
\begin{thm}
The horospherical map is a holomorphic isometric mapping between K\"ahler manifolds $(\frakH\times\R_{>0},\I,g',\Omega')$ and the complex hyperbolic plane $\bH^2_\C$ endowed with the Bergman metric. 
\end{thm}
\begin{proof}
To show that $H$ is holomorphic, we only have to show that $H_*(Z)$ and $H_*(W)$ are in ${\rm T}^{(1,0)}\bH^2_\C$. Here, $Z=\partial_z+i\overline{z}\partial_t$ and $W=(1/2)(\partial_t-i\partial_r)$. Indeed,
\begin{eqnarray*}
H_*(Z)&=&(1/2)Z(-|z|^2-r+it)(\partial/\partial z_1)+(1/2)Z(-|z|^2-r-it)(\partial/\partial \overline{z_1})\\
&&+Z(z)(\partial/\partial z_2)+Z(\overline{z})(\partial/\partial \overline{z_2})\\
&=&-\overline{z_2}(\partial/\partial z_1)+(\partial/\partial z_2)\in{\rm T}^{(1,0)}(\bH^2_\C).
\end{eqnarray*}
Also,
\begin{eqnarray*}
H_*(W)&=&(1/4)(\partial_t-i\partial_r)(-|z|^2-r+it)(\partial/\partial z_1)\\
&&+(1/4)(\partial_t-i\partial_r)(-|z|^2-r-it)(\partial/\partial \overline{z_1})\\
&&+(1/2)(\partial_t-i\partial_r)(z)(\partial/\partial z_2)+(1/2)(\partial_t-i\partial_r)(\overline{z})(\partial/\partial \overline{z_2})\\
&=&(i/2)(\partial/\partial z_1)\in{\rm T}^{(1,0)}(\bH^2_\C).
\end{eqnarray*}
On the other hand, from Eq. (\ref{eq:gyp}) it follows immediately that
$
H^*ds^2=g'.
$
This completes the proof.
\end{proof}
\begin{cor}\label{cor-PCRK2}
The complex hyperbolic plane $\bH^2_\C$ endowed with the Bergman metric K\"ahler cone $(\calC(\frakH),\J, g_r,\Omega_r)$ are {\rm PCR}-K\"ahler equivalent.
\end{cor}

\section{Geodesics} \label{sec-app}
\subsection{Geodesics of the Heisenberg group}
An exhaustive treatment of the geodesics of $(\frakH, g)$ may be found in \cite{Mar}. We repeat in brief this discussion below. Let $\gamma(s)=(x(s),y(s),t(s))$ be a smooth curve defined in an interval $I=(-\epsilon,\epsilon)$, where $\epsilon>0$ and sufficiently small, and suppose that $\gamma(0)=(x_0,y_0,t_0)=p_0$. %Let also
%$$
%X'=(1/\sqrt{2})X,\quad Y'=(1/\sqrt{2})Y.
%$$
The tangent vector $\dot \gamma(s)$ at a point $\gamma(s)$ is then
\begin{eqnarray*}
\dot\gamma(s)=\dot\gamma&=&\dot x\partial_x+\dot y\partial_y+\dot t\partial_t\\
&=& \dot xX+ \dot yY+(1/2)(\dot t+2x\dot y-2y\dot x)\widetilde{T}.
\end{eqnarray*}
We set
$$
f(s)=\dot x(s),\quad g(s)= \dot y(s),h(s)= (1/2)(\dot t(s)+2x(s)\dot y(s)-2y(s)\dot x(s)).
$$
We may assume that $\gamma$ is of unit speed: $f^2+g^2+h^2=1$. The covariant derivative of $\dot\gamma$ is
$$
\frac{D\dot\gamma}{ds}=\dot fX+\dot gY+\dot h\widetilde{T}+f\nabla_{\dot\gamma}X+g\nabla_{\dot\gamma}Y+
h\nabla_{\dot\gamma}\widetilde{T}.
$$
Since
\begin{eqnarray*}
&&
\nabla_{\dot\gamma}X=f\nabla_{X}X+g\nabla_{Y}X+h\nabla_{\widetilde{T}}X
=g\widetilde{T}+hY,\\
&&
\nabla_{\dot\gamma}Y=f\nabla_{X}Y+g\nabla_{Y}Y+h\nabla_{\widetilde{T}}Y=-f\widetilde{T}-hX,\\
&&
\nabla_{\dot\gamma}{\widetilde{T}}=f\nabla_{X}\widetilde{T}+g\nabla_{Y}\widetilde{T}+h\nabla_{\widetilde{T}}\widetilde{T}=fY-gX,
\end{eqnarray*}
we deduce
$$
\frac{D\dot\gamma}{ds}=(\dot f-2gh)X+(\dot g+2fh)Y+\dot h\widetilde{T}.
$$
Therefore, the geodesic equations are
\begin{equation}
\dot f=2gh,\quad \dot g=-2fh,\quad \dot h=0,\quad f^2+g^2+h^2=1.
\end{equation}
In the special case $h=0$, that is, $\gamma$ is horizontal, we obtain the straight lines
\begin{equation}\label{eq:geods}
\gamma(s)=(as+x_0,\;bs+y_0,\;2(ay_0-bx_0)s+t_0),
\end{equation}
where $a,b$ are real constants and $a^2+b^2=1$. Those are all $g_{cc}$-geodesics. %: To see this, observe that
%$$
%\gamma(s)=T_{(a_2,b_2,c_1)}(\gamma^*(s)),\quad \gamma^*(s)=(a_1s,b_1s,0).
%$$ 
%From Eq. (\ref{eq:geods}) we obtain that if $(x_0,y_0,0)$ is a point in $\frakH$ lying on the $xy$-plane, then
%$$
%d_g((0,0,0),(x_0,y_0,0))=\sqrt{2(x_0^2+y_0^2)}.
%$
We now write $F=f+ig$ and $z(s)=x(s)+iy(s)$. Then the above system is written equivalently as
$$
\dot F=-2ic\;F,
$$
and
has general solution
$$
F(s)=ke^{-2ics},\quad k\in\C, \quad |k|^2+c^2=1.
$$
We therefore have $|c|\le 1$. If $|c|=1$, then $k=0$ and
\begin{equation}\label{eq:geodv}
\gamma(s)=(x_0,\;y_0,\;cs+t_0),
\end{equation}
is a vertical geodesics through $p_0$.
%From Eq. (\ref{eq:geodv}) we obtain that if $(0,0,t_0)$ is a point in $\frakH$ lying on the $t$-axis, then
%$$
%d_g((0,0,0),(0,0,t_0))=|t_0|.
%$$
If now $|c|<1$ and $\gamma(s)=(z(s),t(s))$, then 
\begin{eqnarray}\label{eq:geod3}
z(s)&=&\frac{ik(e^{-2ics}-1)}{2c}+z_0,\\\label{eq:geod4}
t(s)&=&\frac{1}{c}\left((1+c^2)s-\frac{(1-c^2)\sin(2cs)}{2c}-\Re(\overline{z_0}k(e^{-2ics}-1))\right)+t_0.
%{\color{blue}\frac{1-c^2}{2c}s-\frac{\sqrt{2}}{2cs}\Re\left((k\overline{\delta})e^{-2ics}\right)+D}\right),
\end{eqnarray}
\subsection{Geodesics of $(\calC(\frakH), \J, g_r, \Omega_r)$}

Let $\gamma(s) = (x(s),y(s), t(s), r(s))$ be a smooth curve in the K\"ahler cone and suppose that $\gamma(0) =(x_0, y_0, t_0, r_0) = q_0$. Its tangent vector is
\begin{eqnarray*}
\dot\gamma(s) = \dot\gamma&=& \dot x\partial_x + \dot y\partial_y + \dot t\partial_t +\dot r\partial_r\\
&=& r\dot xX_r + r\dot yY_r + (r/2)(\dot t+2x\dot y - 2y\dot x )T_r + \dot rR_r.
\end{eqnarray*}
We set
$$
f(s) = r(s)\dot x(s),\quad g(s) = r(s)\dot y(s),\quad h(s) =(1/2)r(s)(\dot t(s)+2x(s)\dot y(s)-2y(s)\dot x(s) ),\quad k(s) =\dot r(s),
$$
and we may suppose that $f^2+g^2+h^2+k^2=1$. The covariant derivative of $\dot\gamma$ is 
\begin{eqnarray*}
\frac{D\dot\gamma}{ds}&=&\dot fX_r+\dot gY_r+\dot hT_r+\dot kR_r\\
&&+f\nabla_{\dot\gamma}X_r+g\nabla_{\dot\gamma}Y_r+h\nabla_{\dot\gamma}T_r+k\nabla_{\dot\gamma}R_r\\
&=&\dot fX_r+\dot gY_r+\dot hT_r+\dot kR_r\\
&&+(f/r)(-fR_r+gT_r+hY_r)\\
&&+(g/r)(-fT_r-gR_r-hX_r)\\
&&+(h/r)(fY_r-gX_r-hR_r)\\
&&+(k/r)(fX_r+gY_r+hT_r).
\end{eqnarray*}
From the vanishing of the covariant derivative and the unit speed assumption we obtain the following geodesic equations:
\begin{eqnarray}
\label{eq-geo1}
\dot f&=&(1/r)(2gh-kf),\\
\label{eq-geo2}
\dot g&=&(1/r)(-2fh-kg),\\
\label{eq-geo3}
\dot h&=&(1/r)(-kh),\\
\label{eq-geo4}
\dot k&=&(1/r)(1-k^2).
\end{eqnarray}
Equation (\ref{eq-geo4}) also reads as
$$
r\ddot r+(\dot r)^2=1.
$$
The positive solutions to this ODE are of the form
$$
r(s)=\sqrt{(s+c_1)^2+c_2},\quad c_1, c_2\in\R,\; c_2\ge0. 
$$
From the initial condition $r(0)=r_0$, we also have $c_1^2+c_2=r_0^2$; thus 
\begin{equation}\label{eq-r}
r(s)=\sqrt{s^2+2c_1s+r_0^2},\;c_1\in\R,\;r_0^2-c_1^2\ge 0.
\end{equation}
From Equation (\ref{eq-geo3}) we obtain
$$
h(s)=\frac{c_3}{r(s)}=\frac{c_3}{\sqrt{s^2+2c_1s+r_0^2}}, \quad c_3\in\R.
$$
But then, from
$$
c_3/r(s)=(1/2)r(s)(\dot t(s)+2x(s)\dot y(s)-2y(s)\dot x(s))
$$
we obtain that
\begin{equation}\label{eq-t0}
\dot t(s)+2x(s)\dot y(s)-2y(s)\dot x(s)=2c_3/(s^2+2c_1s+r_0^2).
\end{equation}
Now, we have 
$$
f^2+g^2=1-h^2-k^2=\frac{r_0^2-c_1^2-c_3^2}{s^2+2c_1s+r_0^2}\ge 0.
$$
In the case where $r_0^2=c_1^2+c_3^2,$ one can get that $f\equiv g\equiv 0$ which yields to
$$
x(s)=x_0,\quad y(s)=y_0.
$$
Also, from $\dot t(s)=2\sqrt{r_0^2-c_1^2}/(s^2+2c_1s+r_0^2)$ we have
$$
t(s)=2\arctan\left(\frac{s\sqrt{r_0^2-c_1^2}}{r_0^2+c_1s}\right)+t_0,\quad r_0^2-c_1^2>0.
$$
In the case where $r_0^2=c_1^2$, it is easy to know that $r(s)=\pm s+r_0$ and $t(s)=t_0$. 
Hence the resulting geodesics in this case are of the form:
\begin{equation}\label{eq-geoext}
\gamma_c(s)=\left(x_0,\;y_0,\;2\arctan\left(\frac{s\sqrt{r_0^2-c^2}}{r_0^2+cs}\right)+t_0,\sqrt{s^2+2cs+r_0^2}\right),\quad c\in\R.
\end{equation}
or straight lines of the form
\begin{equation}\label{eq-geoextlines}
\gamma(s)=(x_0,\;y_0,\;t_0,\;s+r_0).
\end{equation}
%In Cartesian coordinates,
%$$
%x=x_0,\quad y=y_0,\quad r^2=r_0^2\left(\tan^2((t-t_0)/2)+1\right).
%$$
In the case where $r_0^2>c_1^2+c_3^2$, by plugging $r(s),h(s)$ and $k(s)$ into Eqs. (\ref{eq-geo1}) and (\ref{eq-geo2}) we obtain
\begin{eqnarray*}
\dot f&=&\frac{2c_3g-(s+c_1)f}{s^2+2c_1s+r_0^2},\\
\dot g&=&-\frac{2c_3f+(s+c_1)g}{s^2+2c_1s+r_0^2}.
\end{eqnarray*}
We set $F=f+ig$ and the system of geodesic equations becomes the following complex ODE of the first order:
\begin{equation*}
\dot F=-\frac{s+c_1+2ic_3}{s^2+2c_1s+r_0^2} F,
\end{equation*}
where also
$$
|F|^2=\frac{r_0^2-c_1^2-c_3^2}{s^2+2c_1s+r_0^2}>0,
$$
hence,
$$
F(s)=\frac{Ce^{-2i\frac{c_3}{\sqrt{r_0^2-c_1^2}}\arctan\left(\frac{s+c_1}{\sqrt{r_0^2-c_1^2}}\right)}}{\sqrt{s^2+2c_1s+r_0^2}},
$$
where $C \in \mathbb{C}$ satisfying $|C|^2=r_0^2-c_1^2-c_3^2$.
%This gives
%$$
%{\rm Log}(F/C)=-(1/2)\log(s^2+2c_1s+r_0^2)-2ic_3\int\frac{ds}{(s+c_1)^2+r_0^2-c_1^2}, \quad C\in\C_*.
%$$
Since $F(s)=r(s)\dot z(s)$, we have the complex ODE of the first order
$$
\dot z(s)=\frac{Ce^{-2i\frac{c_3}{\sqrt{r_0^2-c_1^2}}\arctan\left(\frac{s+c_1}{\sqrt{r_0^2-c_1^2}}\right)}}{s^2+2c_1s+r_0^2}
$$
which gives
$$
z(s)=\frac{iCe^{-2i\frac{c_3}{\sqrt{r_0^2-c_1^2}}\arctan\left(\frac{s+c_1}{\sqrt{r_0^2-c_1^2}}\right)}}{2c_3}+D,\quad D\in\C.
$$
From the initial conditions we then obtain
\begin{equation}\label{eq-z}
z(s)=\frac{iC}{2c_3}\left(e^{-2i\frac{c_3}{\sqrt{r_0^2-c_1^2}}\arctan\left(\frac{s+c_1}{\sqrt{r_0^2-c_1^2}}\right)}-e^{-2i\frac{c_3}{\sqrt{r_0^2-c_1^2}}\arctan\left(\frac{c_1}{\sqrt{r_0^2-c_1^2}}\right)}\right)+z_0.
\end{equation}
Now from  Eq. (\ref{eq-t0}) we have
\begin{eqnarray*}
\dot t(s)&=&\frac{2c_3^2+|C|^2}{c_3(s^2+2c_1s+r_0^2)}-\frac{|C|^2}{c_3(s^2+2c_1s+r_0^2)}\cos\left(2\frac{c_3}{\sqrt{r_0^2-c_1^2}}\arctan\left(\frac{s\sqrt{r_0^2-c_1^2}}{r_0^2+c_1s}\right)\right)\\
&&-\frac{2}{s^2+2c_1s+r_0^2}\Im\left(\overline{z_0}Ce^{-2i\frac{c_3}{\sqrt{r_0^2-c_1^2}}\arctan\left(\frac{s+c_1}{\sqrt{r_0^2-c_1^2}}\right)}\right).
%\frac{2c_3}{s^2+r_0^2}+\frac{2(r_0^2-c_3^2)}{c_3(s^2+r_0^2)}\sin^2\left((c_3/r_0)\arctan(s/r_0)\right).
\end{eqnarray*}
By integrating and taking under account the initial conditions we obtain
\begin{eqnarray*}
t(s)&=&\frac{2c_3^2+|C|^2}{c_3\sqrt{r_0^2-c_1^2}}\Theta(s)-\frac{|C|^2}{2c_3^2}\sin\left(2\frac{c_3}{\sqrt{r_0^2-c_1^2}}\Theta(s)\right)\\
&&-\Re\left(\overline{z_0}C \left(e^{-2i\frac{c_3}{\sqrt{r_0^2-c_1^2}}\arctan\left(\frac{s+c_1}{\sqrt{r_0^2-c_1^2}}\right)}-e^{-2i\frac{c_3}{\sqrt{r_0^2-c_1^2}}\arctan\left(\frac{c_1}{\sqrt{r_0^2-c_1^2}}\right)}\right) \right)+t_0,%,\quad c_4\in\R
\end{eqnarray*}
where
$$
\Theta(s)=\arctan\left(\frac{s\sqrt{r_0^2-c_1^2}}{r_0^2+c_1s}\right).
$$

\subsection{Non completeness}
The submanifold $\calU=\{(t,r):t\in\R,\;r>0\}$ is totally geodesic. Therefore, in order to prove non-completeness of the metric $g_r$ in $\calC(\frakH)$, it suffices to prove the following:
\begin{prop}\label{geo-2points}
Given two arbitrary distinct points $p_0=(t_0,r_0)$ and $p_1=(t_1,r_1)$ of $\calU.$ 
There exists a geodesic joining these two points if and only if $|t_1-t_0|<2\pi$. Explicitly: 
\begin{enumerate}
\item[(i)] When $t_1=t_0,$ then $p_0$ and $p_1$ are joined by the geodesic $\gamma$ which is the horizontal line $t=t_0$;
\item[(ii)] when $t_1\neq t_0,$ then there exist geodesics $r_a(s)$ pass through the two points, where
\begin{align*}
s=&\pm\sqrt{r_1^2+r_0^2\pm2r_0r_1 \cos\left(\frac{t_1-t_0}{2}\right)},\\
a=&\cfrac{-r_0^2\mp r_0r_1\cos\left(\frac{t_1-t_0}{2}\right)}{\pm\sqrt{r_1^2+r_0^2\pm2r_0r_1 \cos\left(\frac{t_1-t_0}{2}\right)}}.
\end{align*}
\end{enumerate}
\end{prop}
\begin{proof}
 It is sufficient to find the value of $s$ and $a\in(-r_0,r_0)$ such that $\gamma_a(s)=(r_1,t_1)$ when $t_1\neq t_0.$ In particular, from Eg. (\ref {eq-geoext}) we have the equations
\begin{equation}\label{eq-geosol1}
s^2+2as+r_0^2-r_1^2=0
\end{equation}
and 
\begin{equation}\label{eq-geosol2}
t_1-t_0=\pm2\arctan\frac{s\sqrt{r_0^2-a^2}}{r_0^2+as}.
\end{equation}
We observe that $s\neq 0$; otherwise we get from (\ref{eq-geosol1}) that $r_0=r_1$ and from (\ref{eq-geosol2}) that $t_1=t_0$. %Thus we have
It is now clear from (\ref{eq-geosol2}) that if $|t_1-t_0|\ge 2\pi$ then it has no solution. 
Let $\tau=\tan\left(\frac{t_1-t_0}{2}\right).$  It follows form (\ref{eq-geosol2}) that
$$
\tau^2=\cfrac{s^2(r_0^2-a^2)}{(r_0^2+as)^2}.
$$
Moreover, one can get that
$$(r_0^2+ as)^2(1+\tau^2)=r_0^2(r_0^2+2as+s^2),$$
which implies
\begin{equation}\label{geo-s2}
(r_1^2+r_0^2-s^2)^2(1+\tau^2)=4{r_0^2r_1^2},
\end{equation}
because $r_0^2+2as+s^2=r_1^2$ and ${as}=\cfrac{r_1^2-r_0^2-s^2}{2}$ from (\ref{eq-geosol1}).
\bigskip
Therefore, we have $s^2=r_1^2+r_0^2\pm 2r_0r_1\cos\left(\frac{t_1-t_0}{2}\right)$ by (\ref{geo-s2}), i.e.,
$$s=\pm\sqrt{r_1^2+r_0^2\pm2r_0r_1 \cos\left(\frac{t_1-t_0}{2}\right)}.$$
Considering $a=\cfrac{r_1^2-r_0^2-s^2}{2s},$ one can get that
$$a=\cfrac{-r_0^2\mp r_0r_1\cos\left(\frac{t_1-t_0}{2}\right)}{\pm\sqrt{r_1^2+r_0^2\pm2r_0r_1 \cos\left(\frac{t_1-t_0}{2}\right)}}.$$
By a direct calculation, one can check that $a\in (-r_0, r_0).$
\end{proof}
From the above proposition, we obtain from Hopf-Rinow Theorem the following:
\begin{cor}
The totally geodesic submanifold  $\calU$ of $(\calC(\frakH),g_r)$ is not complete. The same holds for $(\calC(\frakH),g_r)$ .  
\end{cor}

\end{document}